\documentclass{amsart}

\usepackage[colorlinks,citecolor=niceblue,linkcolor=niceblue]{hyperref}
\usepackage{comment}
\usepackage{esint}
\usepackage{amssymb, amsrefs}
\usepackage{tikz}
\usetikzlibrary{calc}
\usetikzlibrary{patterns}
\usepackage{graphicx}
\usepackage{enumitem}
\usepackage{array}

\usepackage[export]{adjustbox}

\newtheorem{theorem}{Theorem}
\newtheorem{lemma}[theorem]{Lemma}
\newtheorem{corollary}[theorem]{Corollary}
\newtheorem{proposition}[theorem]{Proposition}
\theoremstyle{definition}

\newtheorem{remark}[theorem]{Remark}

\newtheorem{definition}[theorem]{Definition}

\newtheorem{assumption}[theorem]{Assumption}

\newcommand{\tref}[1]{Theorem \ref{t.#1}}
\newcommand{\lref}[1]{Lemma \ref{l.#1}}
\newcommand{\pref}[1]{Proposition \ref{p.#1}}
\newcommand{\cref}[1]{Corollary \ref{c.#1}}
\newcommand{\fref}[1]{Figure \ref{f.#1}}
\newcommand{\sref}[1]{Section \ref{s.#1}}

\newcommand{\dref}[1]{Definition \ref{d.#1}}

\newcommand{\deflink}[2]{\hyperref[#1]{#2}}

\numberwithin{theorem}{section}
\numberwithin{equation}{section}

\newcommand{\R}{\mathbb{R}}

\newcommand{\grad}{\nabla}

\def\XXint#1#2#3{{\setbox0=\hbox{$#1{#2#3}{\int}$ }
\vcenter{\hbox{$#2#3$ }}\kern-.6\wd0}}

\newcommand{\Diss}{\operatorname{Diss}}

\newcommand{\one}{\mathbf{1}}

\definecolor{niceblue}{rgb}{0,0,0.7}

\def\strikethrough#1{\setbox0\hbox{#1}\rlap{#1}\hbox to \wd0{\hss\strikebox\hss}}
\def\strikebox{\vrule height 0.6\ht0 depth -0.4\ht0 width 1.1\wd0}

\begin{document}
\title[A monotone model of droplet motion]{Minimizing movements solutions for a monotone model of droplet motion}
\author{Carson Collins}
\address{Department of Mathematics, University of California, Los Angeles, California, 90095, USA}
\author{William M Feldman}
\address{Department of Mathematics, University of Utah, Salt Lake City, Utah, 84112, USA}
\begin{abstract}
    We study the uniqueness and regularity of minimizing movements solutions of a droplet model in the case of piecewise monotone forcing. We show that such solutions evolve uniquely on each interval of monotonicity, but branching non-uniqueness may occur where jumps and monotonicity changes coincide.  This classification of minimizing movements solutions allows us to reduce the quasi-static evolution to a finite sequence of elliptic problems and establish $L^\infty_tC^{1,1/2-}_x$-regularity of solutions.
\end{abstract}
\maketitle

\section{Introduction}
We consider an energetic model for the rate-independent evolution of capillary droplets under the effects of contact angle hysteresis.   This type of model was introduced by DeSimone, Grunewald, and Otto in \cite{DeSimoneGrunewaldOtto}, the existence theory was developed in \cite{alberti2011}, and regularity and uniqueness results were proved under (piecewise) monotonic forcing and specific geometric hypotheses in \cite{FeldmanKimPozar}.  In this paper we will establish a uniqueness and regularity theory in general geometry, but still using (piecewise) monotonic forcing in an essential way.

In this model, the free surface of a droplet above a smooth flat surface is considered as the graph of a non-negative $H^1$ function $u$ on a domain $U \subset \R^d$. The positivity set $\Omega(u) = \{ u > 0 \}$ is the liquid-solid contact region. The energy will be the Alt-Caffarelli one-phase functional:
    \begin{equation}
        \mathcal{J}[u] = \int_U |\nabla u|^2 + \one_{\{u > 0\}} =: \mathcal{D}[u] + |\Omega(u)|
    \end{equation}
    Here, the Dirichlet energy $\mathcal{D}[u]$ is the linearization of the surface area of the free surface.
    
    The second piece of the model is a dissipation distance.  This distance measures the energy dissipated due to frictional forces as the contact line moves from one state to another.  The dissipation distance between two contact regions $V_0$ and $V_1$, sets in $\R^d$, is defined
    \begin{equation}
        \Diss[V_0, V_1] = \mu_+|V_1 \setminus V_0| + \mu_-|V_0\setminus V_1|
    \end{equation}
   Here, $\mu_-, \mu_+$ are constant parameters representing the frictional force per unit length as the droplet recedes or advances respectively. We will require $0 < \mu_- < 1$ and $0 < \mu_+$. We will often write $\Diss[u_0, u_1]$ to mean $\Diss[\Omega(u_0), \Omega(u_1)]$.

   We consider a quasistatic evolution of states $u(t)$ associated with this energy functional and dissipation distance. In this setting, the time scale of variations of the forcing is assumed to be much slower than the time scale for viscous regularization of the free surface shape so that the droplet is always in local equilibrium.  Specifically, the graph of $u(t)$ is harmonic, and the ``force'' per unit length, $|\nabla u|^2 - 1$, on the interface coming from the first variation of free energy is within the pinning interval $[\mu_-,\mu_+]$, i.e. it is bounded by the coefficient of static friction for receding and advancing contact lines respectively. As is common in rate-independent literature \cite{mielke2015book}, we strengthen local equilibrium to a global stability condition: at every time $u(t)$ minimizes the energy / dissipation functional based at $u(t)$
\begin{equation}
    v \mapsto \mathcal{E}[u(t),v]:=\mathcal{J}[v]+\Diss[u(t),v]
\end{equation}   
over an appropriate admissible class.  

     The evolution is driven by an external forcing.  A natural physical choice of forcing, corresponding to evaporation or condensation processes, is a time varying volume constraint. Instead, as in \cite{FeldmanKimPozar}, we consider a Dirichlet driving force
    \[ u(t,x) = F(t,x) \ \hbox{ on } \ \partial U.\]
    The Dirichlet forcing problem has improved monotonicity and uniqueness properties, and can be viewed as a localization of the volume constrained problem.

     Our paper will consider minimizing movements solutions, which are constructed as the limits of the time incremental scheme
   \[ u(t_k)\in \mathrm{argmin}\{ \mathcal{J}[v] + \Diss[u(t_{k-1}), v] : v|_{\partial U} = F(t_k) \} \]
    In this, $F(t,x)$ is a time-varying Dirichlet forcing which drives the quasi-static evolution of the droplet.
   
    Our main result is a uniqueness property of minimizing movements solutions in the case of piecewise monotone forcing.  As a corollary, applying regularity results on Bernoulli obstacle problems from the literature \cites{ChS,FerreriVelichkov}, we can also show the (almost) optimal regularity $\textup{L}^\infty_tC^{1,\frac{1}{2}-}_x$ in $d \leq 4$.

\subsection{Details of the model}

We now make precise the notion of minimizing movements solution that we will work with in the paper.

\begin{definition}\label{def:mmsln}\label{d.mmsln}
    Let $P_\delta$ be a system of $\delta$-fine partitions of $[0,T]$. A \emph{minimizing movements scheme} for the system of partitions is a sequence $u^{t_k}_\delta, t_k\in P_\delta$ such that
\begin{equation}
    u^{t_k}_\delta \in \mathrm{argmin} \{ \mathcal{E}[u^{t_{k-1}}_\delta, u'] : u'\in H^1_0(U) + F(t_k)\}
\end{equation}
For concision, we write
\begin{equation}\label{eq:argmin_notation}
    \mathcal{M}[u, F] = \mathcal{M}[\Omega(u), F] := \mathrm{argmin} \{ \mathcal{E}[u^{t_{k-1}}_\delta, u'] : u'\in H^1_0(U) + F(t_k)\}
\end{equation}
We note that $\mathcal{M}[u, F]$ is a set which may contain multiple minimizers in general, due to the nonuniqueness of the exterior Bernoulli problem. This is closely tied to the ability of the evolution to jump.

Define $u_\delta(t,x)$ by discontinuous interpolation in time: $u_\delta(t) = u^{t_k}_\delta$ when $t\in [t_{k}, t_{k+1})$. It is then known that there exists a sequence $\delta_k\to 0$ such that for each $t$, $u_{\delta_k}(t)$ converges uniformly in $x$ as $k\to \infty$. A \emph{minimizing movements solution} is any such subsequential limit of the $u_\delta$ as $\delta\to 0$.
\end{definition}

Generally speaking, in the literature, minimizing movements schemes as in Definition \ref{def:mmsln} are used to construct energy solutions in the following sense: 

\begin{definition}\label{def:energy_sln}
    A measurable $u : [0,T] \to H^1(U)$ is an \emph{energy solution} of the dissipative evolution driven by forcing $F$ if
    \begin{enumerate}
        \item (Forcing) For all $t \in [0,T]$
        \[u(t,x) = F(t,x) \ \hbox{ on } \ \partial U.\]
        \item (Global stability) The solution $u(t)$ satisfies for all $t \in [0,T]$
        \[\mathcal{J}[u(t)] \leq \mathcal{J}[u'] + \Diss[u(t),u'] \ \ \hbox{ for all } \ u'  \in F(t) + H^1_0(U)\]
        \item (Energy dissipation inequality) For all $0 \leq t_0 \leq t_1 \leq T$ it holds
        \[\mathcal{J}[u(t_0)] - \mathcal{J}[u(t_1)] + \int_{t_0}^{t_1}\int_{\partial U} 2F(x,t)\partial_\nu u(x,t)  \ dx dt \geq \Diss[u(t_0),u(t_1)].\]
    \end{enumerate}
\end{definition}

We remark that the dissipation inequality for energy solutions becomes equality, if one replaces the right hand side with the \textit{total dissipation}; see \cite{FeldmanKimPozar}. On the other hand, we refer to \cite{alberti2011} for an example with $u(t)$ only locally minimizing $\mathcal{E}[u(t), \cdot]$ for which the dissipation inequality is strict at a jump.

We show below, in \tref{mm-energy}, that minimizing movements solutions are indeed energy solutions.  However, we are currently unable to classify all the energy solutions, see \sref{open-q} for more discussion.

\subsection{Main result}
Our main result is on the uniqueness of minimizing movements solutions.  First we show, in the case of monotonically varying forcing, that there is only one minimizing movements solution modulo redefinition at jump times. Moreover, any time discrete scheme will simply sample from this solution at the corresponding times.

Next we employ these ideas in piecewise monotone case. However now an important complication arises.  If the monotone evolution has a jump at the exact time that the forcing changes monotonicity, then there will be a solution which takes the jump and a solution which does not take the jump. Due to friction, a solution which jumps up does not immediately jump back down when the forcing is reversed (see \lref{diff_ics_merge} below), and so the limiting solutions should not be the same in general. A depiction of several such scenarios can be found in \fref{mon-change-example}.  Our main result is that this is the only source of non-uniqueness for minimizing movements solutions.

\begin{definition}
    We say that $F(x,t)$ is strictly increasing if, for any $s < t$, we have $F(x, s) \leq F(x, t)$ for all $x\in \partial U$, and moreover, on each component of $\partial U$ there exists a $y$ such that $F(y, s) < F(y, t)$. We say that $F$ is strictly monotone if it is strictly increasing or strictly decreasing.
\end{definition}

\begin{theorem}\label{t.main}
Let $0 = t_0 < \dots < t_N = T$ and suppose that $F(t, x)$ is strictly monotone in $t$ on each $[t_j, t_{j+1}]$. We also adopt basic regularity assumptions on $U$, $F$, and $\Omega_0$ listed in \ref{assumptions}.

Then any sequence $u(t_i)$, chosen recursively by $u(t_i)\in \mathcal{M}[u(t_{i-1}), F(t_i)]$, defines a minimizing movements solution at the intervening times via
\begin{equation}\label{eq:thm_eqn}
    u(t)\in \mathcal{M}[u(t_i), F(t)] \qquad t \in (t_i, t_{i-1}).
\end{equation}
This is a genuine definition of $u$, in the sense that all solutions with the same $u(t_i)$ jump at the same times and agree up to value at jumps. Conversely, all minimizing movements solutions have the form \eqref{eq:thm_eqn} and are piecewise monotone on each $[t_j, t_{j+1}]$. All minimizing movements solutions are right continuous at monotonicity changes, and there is a unique left continuous minimizing movements solution, characterized by not jumping at monotonicity changes.
\end{theorem}

We stress that solutions may genuinely branch based on jumps at monotonicity changes (see \fref{mon-change-example}). 

This uniqueness result reduces the minimizing movements evolution to a finite sequence of elliptic Bernoulli obstacle problems. Then we can apply the regularity results of Chang-Lara and Savin \cite{ChS} and Ferreri and Velichkov \cite{FerreriVelichkov} to derive the following (almost) optimal result on the space-time regularity of minimizing movements solutions.

\begin{corollary}\label{c.low_dim_regularity}
    For any $ 0 < \beta < \frac{1}{2}$, if $F\in L^\infty([0, T]; C^{1,\beta}(\partial U))$ is piecewise monotone with $N$ monotonicity changes, then in dimension $d \leq 4$, we have $u(t)\in C^{1,\beta}(\overline{\Omega(u(t))})$, with an estimate uniform in time depending on all parameters and $N$.
\end{corollary}

 A similar result is possible in higher dimensions, with additional technicalities to handle singular points for the Bernoulli one-phase problem.  The statement in higher dimensions, and the proofs can be found below in \sref{bernoulli-obs-reg}.
 
It is unclear if this result is also true at the endpoint $\beta = 1/2$.  R\"uland and Shi \cite{RulandShi}*{Remark 3.11} observed that solutions of the thin obstacle problem with $C^{1,1/2}$ obstacle may lose a logarithm in the modulus of continuity of the derivative.  Since our result is based on repeated solutions of Bernoulli obstacle problems we would need to carefully analyze the the obstacle regularity with these logarithmically corrected moduli to obtain a result at $\beta = 1/2$.

\subsection{Literature}

 Our model is derived from a model for contact angle hysteresis introduced by \cites{DeSimoneGrunewaldOtto,alberti2011}. The model considers a slowly driven capillary droplet, with the standard capillary free energy based on the total interfacial areas.  Contact angle hysteresis is introduce via the dissipation distance. This is a phenomenological model of the effects of contact line pinning which matches known behaviors but has not been derived directly from a microscopic model.  Due to the slow driving, inertial and viscous effects are ignored, and the evolution is determined entirely by the forcing, energetics, and frictional dissipation distance. Said forcing, in \cite{alberti2011}, takes the form of a varying volume constraint, corresponding to adding or removing liquid from the droplet over time. The theory for the model with this constraint was developed further and generalized to $d\geq 3$ in \cite{alberti2011}, which provides existence results, many examples, and deeper explanations of the physics and mathematics involved.

The case of Dirichlet forcing in place of volume constraint was first considered in \cite{FeldmanKimPozar}. This type of forcing provides many mathematical advantages via monotonicity and comparison properties. The Dirichlet constraint is relevant physically in the case when one considers the local behavior near the contact line taking the macroscopic behavior as a given to provide the external Dirichlet driving for the local behavior. One particular scenario is the very low velocity Wilhelmy plate method where a plate or fiber is lifted slowly from a liquid in order to measure dynamic contact angles, see for example \cite{SauerCarney}. In the reference frame of the moving plate, the height of the bulk free surface could be approximated by a Dirichlet condition.  The paper \cite{FeldmanKimPozar} considered, as we do here, the Bernoulli one-phase energy functional instead of the capillarity energy. Although the Bernoulli energy with Dirichlet forcing has simplified twice from the capillarity energy with volume constraint, it serves as a useful toy model, which sees the nonlinear and non-graphical nature of the PDE determining the contact line. The paper \cite{FeldmanKimPozar} also introduces a notion of \textit{obstacle solution} for piecewise monotone forcing based on minimal supersolutions and maximal subsolutions. This type of solution features a different jump law than energy solutions, jumping ``as little and as late as possible". This is in contrast with the (global) energy solutions we consider here, which typically jump ``as early as possible" before topological singularities occur. Under strong assumptions of star-shapedness for the initial data, that paper shows that obstacle and energy solutions coincide, are unique, and have $C^{1,1/2-}$ free boundary regularity. Thus, in comparison to that paper, we are able to significantly weaken the hypotheses on the domain geometry, but we remain unable to handle the non-monotone evolution.

Both the functional $\mathcal{J}[u]$ and the augmented functional $\mathcal{E}[\Omega_0, u]$ lead to Bernoulli one-phase free boundary problems where the minimizer both vanishes on the free boundary and satisfies a Neumann condition there. To make this connection clear, we note that $\mathcal{E}$ can be rewritten as
\[ \mathcal{E}[\Omega_0, u] = \int_U |\nabla u|^2 + \left((1 - \mu_-)\one_{\Omega_0} + (1 + \mu_+)\one_{U\setminus \Omega_0}\right)\one_{\{ u > 0 \}}. \]

The main difference between the free boundary problems for $\mathcal{J}[u]$ and $\mathcal{E}$ comes from the jump discontinuity in the coefficients of the augmented functional, which is what allows for pinning. The minimizers of $\mathcal{J}[u]$ have been well-studied since work of Alt and Caffarelli \cite{AltCaffarelli}.  For example, it is known that the free boundary of minimizers is smooth outside of a singular set of dimension at most $d-5$, see below in \sref{bernoulli-obs-reg} for further discussion. The one-phase problem with jump discontinuity along a manifold has been studied more recently, first in \cite{ChS} and later in \cite{FerreriVelichkov}. These sources show that near the jump set, the free boundary behaves like a solution to the thin obstacle problem, and so the sharp $C^{1,1/2}$ regularity of the thin obstacle problem limits the regularity of the free boundary problem. We remark that this is a comparatively rough thin obstacle problem, since for typical uses of $\mathcal{E}$, the jump set $\Omega_0$ is the free boundary at a previous time step. The thin obstacle problem with $C^{1,\alpha}$ obstacle is studied in \cite{RulandShi}, where they show that, for a $C^{1,1/2}$ obstacle, the resulting free boundary estimate should be $C^{1,1/2}$ with logarithmic correction. We refer to that paper for further details regarding the endpoint, and we report our free boundary regularity as $C^{1,1/2-}$ for simplicity. 

The theory of rate-independent motion was developed to approach phase transitions with hysteresis, especially dry friction, plasticity, and fracture \cites{MT1,MT2}.  There is a large literature of such systems, and we refer to \cite{mielke2015book} for a more thorough review. There are only a few works on uniqueness and higher spatial regularity, and the existing methods seem to be inapplicable to droplet motion problems.    The paper \cite{RindlerSchwarzacher} studies the case of convex energy functional, which has uniqueness and time regularity since convexity rules out jumps. This is unavailable to us since the volume term in $\mathcal{J}$ is highly non-convex. Other papers do handle uniqueness with jumps \cite{Mach}, but with a regular energy in a finite dimensional state space. The work \cite{RossiStefanelliThomas} considers a monotonically forced model for delamination associated with the perimeter functional. There the passage to the limit in minimizing movements schemes to attain the energy dissipation balance is more challenging, and specific examples with $C^1$ regularity of the delamination are discussed.  Our work, using monotonicity in a key way to describe the uniqueness and higher regularity of an energetic evolution without specific geometric hypotheses, appears to be new in the literature.  

\subsection{Open questions and challenges}\label{s.open-q}

We contrast our findings for the evolution with Dirichlet forcing against what is known for the volume-constrained evolution. Minimizing movements for the latter can exhibit many more types of nonuniqueness, even for spherical profiles and a monotonically varying constraint. For example, given a profile at time $t$ which is fully depinned, one should expect increasing volume to lead to a profile at time $t + \delta$ which compactly contains the original profile (as the entire boundary moves). But in fact, any translation of the profile at time $t + \delta$ which contains the profile at time $t$ has the same dissipation cost and is equivalent in the view of energy minimization. By means of such translations, minimizing movements solutions to the volume-constrained problem may drift apart. The key difference with our problem seems to be the rigidity of the fixed boundary $\partial U$, which enables us to use comparison arguments. It is interesting to consider whether our results should suggest a similar generic uniqueness for the volume-constrained model, up to translation of components or under additional constraints that restrict drift.

\begin{figure}
    \centering
    \begin{tabular}{ c }
         \includegraphics[scale=0.2]{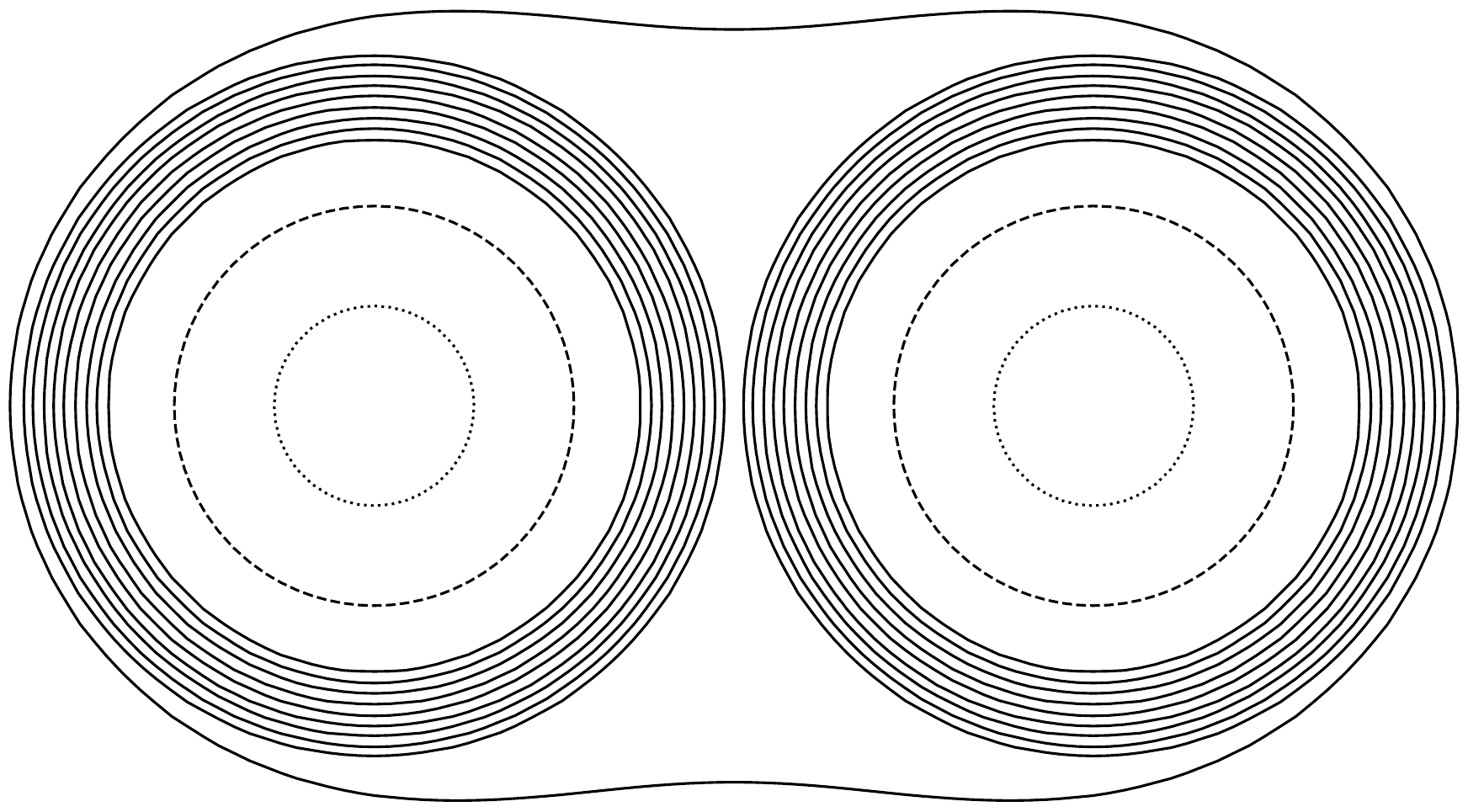}
    \end{tabular}
    \begin{tabular}{ c }
         \includegraphics[scale=0.2]{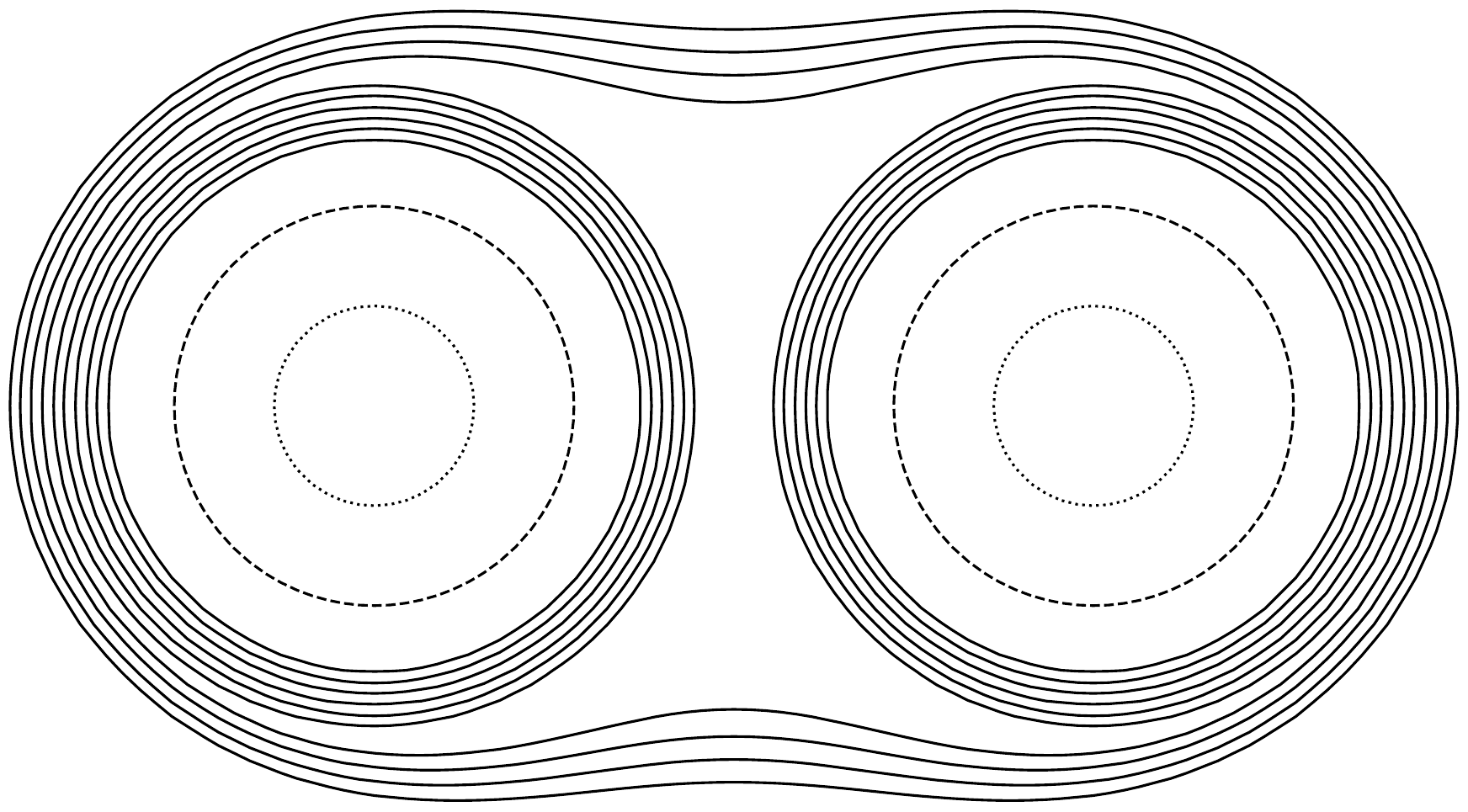}  \\
         \includegraphics[scale=0.2]{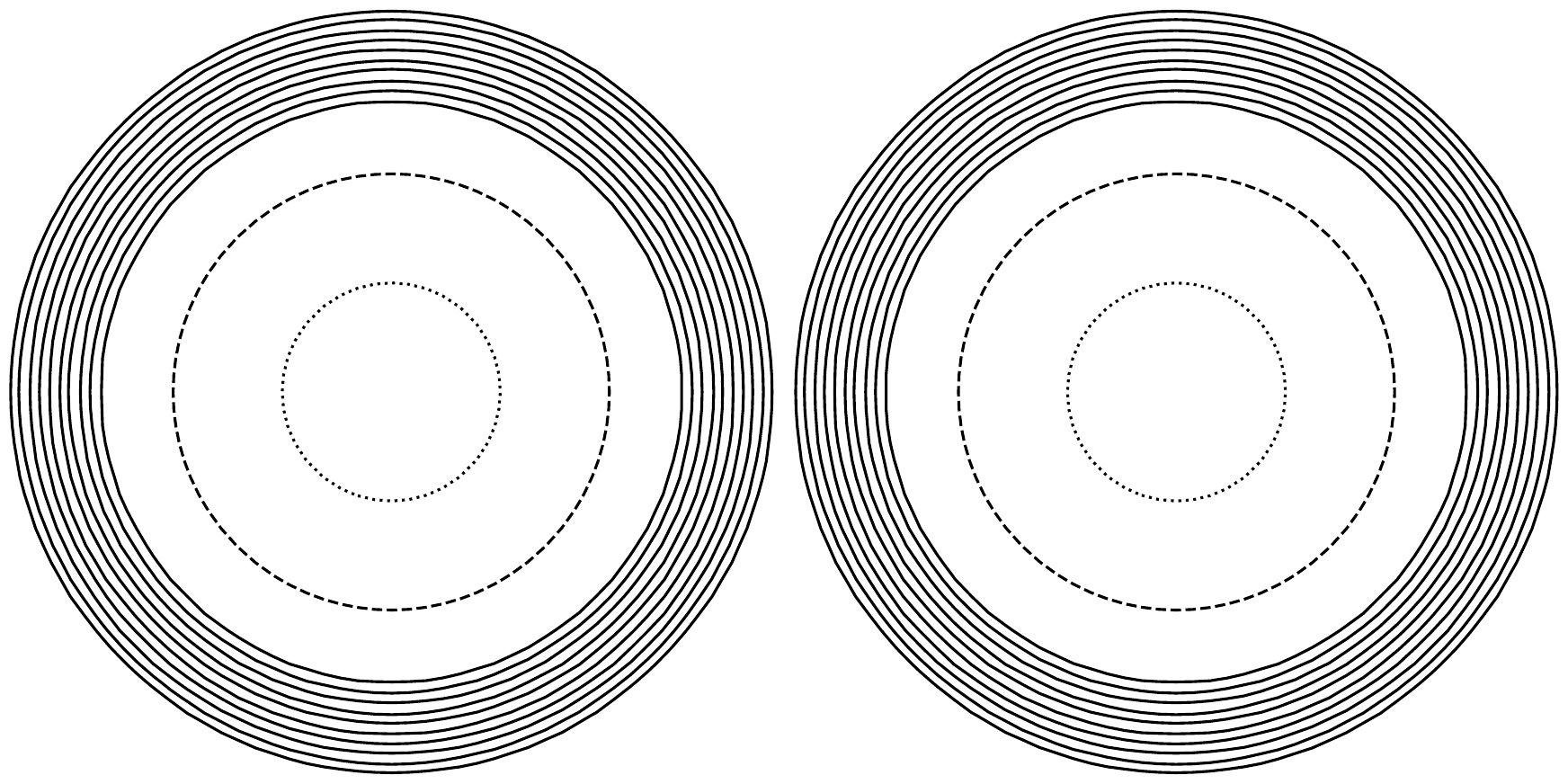}
    \end{tabular}

    \begin{tabular}{cc}
        $0 \leq t \leq t_1$ \hspace{.15\textwidth} &  \hspace{.18\textwidth}$t_1 \leq t \leq t_2$
    \end{tabular}

    \caption{Example of non-uniqueness due to monotonicity change at jump time. Dotted lines indicate Dirichlet boundary $\partial U$. Dashed lines indicate $\partial \Omega_0$. Solid lines indicate $\partial \Omega(t)$. Forcing $F(t)$ is constant in $x$, increasing on $[0, t_1]$, and decreasing on $[t_1, t_2]$.   
    \emph{Left}: Forcing increases until the disconnected annuli jump to a merged profile. At the jump time, the merged profile is energetically equivalent to the two annuli, so the evolution may choose either profile.
    \emph{Top right}: Continuation of the evolution from the left, assuming the jump is taken. Note that due to hysteresis, the profile de-pins and slides before jumping down to an unmerged profile.
    \emph{Bottom right}: Continuation of the evolution from the left, assuming the jump is not taken.}
    \label{f.mon-change-example}
\end{figure}

With piecewise monotone forcing, there can be many different minimizing movements solutions continuing from a monotonicity change, as seen in \fref{mon-change-example}. We expect that this branching non-uniqueness is generically binary. However one can contrive multiple merging events to happen at the same time, possibly leading to a larger number of branching non-unique solutions. 

Even in the context of monotone forcing we are unable to fully classify the energy solutions.  The minimality in the energy dissipation inequality is hard to use due to a lack of control of the energy flux term.  Definitely in the piecewise monotone case there are energy solutions which are not minimizing movements solutions in the sense of \dref{mmsln}. By \tref{main} minimizing movements solutions are right continuous at monotonicity changes, but by \cite{FeldmanKimPozar}*{Proposition 5.16} the left continuous version would be a solution as well.  This seems to be a minor discrepancy though due to the particular definition of minimizing movement solution that we use here, and we are unsure if there are seriously distinct types of energy solutions.  This connects to the difficulty in \cite{FeldmanKimPozar}[Theorem 1.6] where the dynamic slope condition can only be justified in the almost everywhere sense.

\subsection{Acknowledgments}  Both authors were partially supported under the NSF grant DMS-2407235. C. Collins was also supported under the NSF grant DMS-2153254. Both authors would like to thank Inwon Kim for many helpful conversations, and Norbert Po\v{z}\'ar for helpful conversations and providing and modifying his code base to generate the images in \fref{mon-change-example}.

\section{Preliminaries}

We start with the basic assumptions which will be in force throughout the paper.

\begin{assumption}\label{assumptions}~
\begin{enumerate}[label = (\roman*)]
    \item $U\subsetneq \R^d$ is a $C^2$ open set with compact boundary.
    \item $F \in \mathrm{BV}([0, T]; H^{1/2}(\partial U)) \cap C([0, T]; H^{1/2}(\partial U) \cap C(\partial U))$.
    \item There exists $\varepsilon_0 > 0$ such that $F(x, t) \geq \varepsilon_0$ for all $x\in \partial U$ and $t\in [0, T]$.
    \item The initial data is stable, $u_0\in \mathcal{M}[u_0, F(0)]$ (otherwise, there will be a jump at time $0$).
\end{enumerate}
\end{assumption}

We record some standard facts for energy minimizers which are proved in \cite{FeldmanKimPozar}. Some elements of the proof of the below lemma require minor changes to handle $x$-dependence in the forcing. We detail these changes in Appendix \ref{a.energy-solution-properties}.

\begin{lemma}\label{lem:basic_estimates}
Suppose $u$ is an energy solution as in Definition \ref{def:energy_sln}. Then, we have the following estimates:
\begin{enumerate}[label = (\roman*)]
    \item (\cite{FeldmanKimPozar}, Lemma 5.7) \[\mathcal{D}[u(t)] \in \mathrm{BV}([0, T]; \R),\;\; \one_{\Omega(u(t))}\in \mathrm{BV}([0, T]; L^1(\R^d))\] Hence, $u$ is bounded in $H^1(U)$, uniformly in time.
    \item\label{lem:basic_lipschitz} (\cite{FeldmanKimPozar}, Lemma 5.5(i)) If $B_2 \subset U$, then 
    \[\|\nabla u(t)\|_{L^\infty(B_1)} \leq C(d)(\sqrt{1 + \mu_+} + \|\nabla u(t)\|_{L^2(B_2)})\]
    Hence, $u$ is uniformly Lipschitz near the free boundary, as long as the free boundary stays away from $\partial U$.
    \item (\cite{FeldmanKimPozar}, Lemma 5.5(ii), linear nondegeneracy) If $B_r(x) \subset U$ and $x\in \partial \{ u(t) > 0 \}$, then \[\sup_{B_r} u \geq c(d, \mu_-) r.\] 
    \item (\cite{FeldmanKimPozar}, Lemma 5.16) $u(t)$ has left and right limits at every time in the uniform metric, and $\Omega(u(t))$ has left and right limits at every time in the Hausdorff metric.
    \item (\cite{FeldmanKimPozar}, Lemma 5.16) At every time $u(t)\in \mathcal{M}[u(t-), F(t)]$, and $u(t+) \in \mathcal{M}[u(t), F(t)]$. Moreover, $u(t-)\wedge u(t+), u(t-)\vee u(t+)$ are harmonic on their respective positivity sets, hence $u(t-), u(t+)$ are ordered on each component of $\Omega(u(t-))\cup\Omega( u(t+))$.
\end{enumerate}
\end{lemma}

We need a few more results to show that the energy solutions avoid various kinds of pathological behavior; the energy cannot blow up, the profile cannot collapse to the fixed boundary, and the profile cannot extend infinitely. To our knowledge, these results have not previously been written.

\begin{lemma}\label{lem:domain_size_estimate}
Let $u$ be a globally stable profile with forcing $F\in H^{1/2}(\partial U)$, satisfying $F \geq \varepsilon > 0$. We have the following:
\begin{enumerate}[label = (\roman*)]
    \item $\mathcal{J}[u] \leq C(d, \partial U)(\|F\|_{H^{1/2}(\partial U)}^2 + 1)$
    \item There exists $\delta = \delta(d, \partial U, \varepsilon) > 0$ such that $\mathrm{dist}(\partial \Omega(u), \partial U) > \delta$.
    \item There exists $R = R(d, \partial U, \|F\|_{H^{1/2}(\partial U)}) < \infty$ such that $\mathrm{diam}(\Omega(u)) < R$.
\end{enumerate}
\end{lemma}
\begin{proof}
To see the upper bound on $\mathcal{J}[u]$, note that for any alternative profile $u'\in H^1_0(U) + F$, global stability gives 
\[\mathcal{J}[u] \leq \mathcal{J}[u'] + \Diss[u, u'] \leq \mathcal{J}[u'] + \mu_-|\Omega(u)| + \mu_+|\Omega(u')| \]
Then absorbing the measure terms into each $\mathcal{J}$ results in
\[ \mathcal{J}[u] \leq \frac{1 + \mu_+}{1 - \mu_-}\mathcal{J}[u'] \]
We can then take $u'$ to be the solution to the Dirichlet problem on a fixed subset of $U$ so that the Dirichlet energy of $u'$ is proportional to $\|F\|_{H^{1/2}(\partial U)}^2$ based on the chosen subset, while the volume term of the energy is independent of $F$.

For the second part, we use \cite{FeldmanKimPozar}*{Lemma A.2} to estimate the energy difference associated to filling in holes near the fixed boundary. The formula gives that if $v_0 \leq v_1$ on $U$, the two functions have the same trace on $\partial U$, and $v_1$ is subharmonic, then
\begin{equation}\label{eq:energy_difference_quotient}
    \mathcal{D}[v_0] - \mathcal{D}[v_1] \geq \int_{\Omega(v_1)\setminus \Omega(v_0)} |\nabla v_1|^2
\end{equation}
In particular, we let $v_0 = u$, and $v_1$ be the pointwise maximum of $u$ with an appropriately chosen barrier $\Phi$. So long as $\Phi$ is subharmonic and we have $|\nabla \Phi|^2 > 1 + \mu_+$, \eqref{eq:energy_difference_quotient} contradicts global stability of $u$ wherever $u = 0$ and $\Phi > 0$.

We build $\Phi$ as a fundamental solution on an annulus $B_{r_2}\setminus B_{r_1}$, where $B_{r_1}\subset U^c$. We take $\Phi = \varepsilon/2$ on $\partial B_{r_1}$ and $\Phi = 0$ on $\partial B_{r_2}$, so that $u\vee \Phi = u$ on $\partial U$ and 
\[|\nabla \Phi|^2 \geq  C(d)\varepsilon^2\left(1 - \left(r_1/r_2\right)^{d-1}\right)^{-1}.\] When $r_1/r_2$ is sufficiently close to 1, we get the estimate needed for \eqref{eq:energy_difference_quotient}.  Since $U$ is outer regular, each $x_0 \in \partial U$ has an interior touching ball of radius $r_1>0$ and, choosing $r_2$ a constant multiple of $r_1$ based on the previous argument, shows that $u>0$ in a neighborhood $B_{r_2-r_1}(x_0) \cap U$.

Turning to the last part, we observe that the Lipschitz estimate and linear nondegeneracy estimate at the free boundary from Lemma \ref{lem:basic_estimates} imply there exist constants $c, r_0 > 0$ such that $|B_r(x) \cap \Omega(u)| \geq c r^d$ for $r < r_0$ and $x\in \partial \Omega(u)$. Then, let us choose $r$ in this range and cover $\partial \Omega(u)$ and $\{ x \in \Omega(u) : \mathrm{dist}(x, \partial\Omega(u) \geq r\}$ by balls of radius $r$. We can choose the collection of balls such that only $C(d)$ balls intersect at any given point. Ignoring technicalities near $\partial U$, each ball in the covering intersects $\Omega(u)$ in at least $cr^d$ measure, so the total number of balls is at most $\frac{C(d)|\Omega(u)|}{c r^d}$. A competitor argument yields that all connected components of a stable profile border at least one component of $\partial U$, so it follows that $\mathrm{diam}(\Omega(u)) \leq \frac{C(d)|\Omega(u)|}{c r^{d-1}} + \mathrm{diam}(\partial U)$. We recall that $|\Omega(u)|, c, r$ are all controlled by $\mathcal{J}[u]$, which we have already bounded, so we conclude.

\end{proof}

\subsection{Regularity results for Bernoulli obstacle problems}\label{s.bernoulli-obs-reg} In this section we collect and slightly refine some results from the literature on the regularity of Bernoulli obstacle problems. Then we show how to apply these results in our context using the uniqueness theorem \tref{main}. For further results and references on the regularity theory beyond this brief discussion, we refer to the book of \cite{VelichkovBook}.

The regularity theory of the Bernoulli obstacle problem is developed through linear blowup at the free boundary; that is, studying subsequential limits of $v_r(y) = r^{-1} u(x_0 + ry)$ as $r\to 0$, for $x_0 \in \partial \Omega(u)$. The Lipschitz estimate of Lemma \ref{lem:basic_estimates} gives the compactness of the blowup sequence, while the nondegeneracy estimate ensures that subsequential limits are nontrivial and minimize $\mathcal{J}$ with respect to compact perturbations. Furthermore, the Weiss monotonicity formula can be used to show that all subsequential limits are (positively) 1-homogeneous, see \cite{FerreriVelichkov}*{Section 8}.

Free boundary points are classified into regular or singular points based on their blowup limits. Regular points have blowup limits of the form $\alpha(y\cdot \nu)_+$, called half-planes; in this case, ``flat implies $C^{1,\beta}$" estimates can be used to show regularity of the free boundary in a neighborhood of the regular point. In particular, regular points are relatively open in the free boundary, and the $\nu$ in the blowup is the inward normal vector to the free boundary. A point with any other type of blowup limit is a singular point, and the free boundary cannot be $C^1$ at such a point. In dimension $d\leq 4$, it has been shown in \cites{AltCaffarelli,CaffarelliJerisonKenig,JerisonSavin} that singular points do not occur. In dimension $d = 7$, \cite{DeSilvaJerison} shows the existence of singular points. The smallest dimension $d_{\text{sing}}$ in which singular points exist is presently unknown, so we may only say that $d_{\text{sing}} \in \{ 5, 6, 7 \}$. In general, for $d\geq d_{\text{sing}}$ the set of singular points in the free boundary is at most a closed set of Hausdorff dimension $d - d_{\text{sing}}$, see \cite{Weiss}.

\begin{theorem}[Chang-Lara and Savin \cite{ChS}, Ferreri and Velichkov \cite{FerreriVelichkov}]\label{t.regularity_estimate}
Let $u\in \mathcal{M}[u_0, F]$ and $ \beta \in (0,\frac{1}{2})$.  Suppose that $\Omega_0 := \Omega(u_0)$ is locally $C^{1,\beta}$ away from a closed set $K_0 \subset \partial \Omega_0$, and if $u \geq u_0$ or $u \leq u_0$. Then $\Omega(u)$ is locally $C^{1,\beta}$ away from $K_0\cup K_1$, where $K_1$ is a closed set of dimension $d - d_{\textup{sing}}$ which is disjoint from $\partial \Omega_0$. In particular $K_1 = \emptyset$ for $d \leq 4$.
\end{theorem}
\begin{proof}
We start by explaining how most of this statement follows from known results in the literature.  The main point which we will need to clearly justify, is that the ``new" singular set $K_1$ is disjoint from $\partial \Omega_0$.
    
     By the flat implies $C^{1,\beta}$ results of \cites{ChS,FerreriVelichkov} it suffices to show that $u$ has a half-plane blow-up at every $x_0 \in \partial \Omega_0 \setminus K_0$.  The idea is to apply \cite{CaffarelliSalsa}*{Lemma 11.17}, which shows that positive harmonic functions have a non-tangential derivative at inner and/or outer regular points of the boundary of their positivity set. For the case $u \leq u_0$ see the comments at the beginning \cite{ChS}*{Section 2.2}. For the case $u \geq u_0$ the argument is slightly more complicated. Following the approach to blowup analysis described above, up to a subsequence,
    \[\lim_{r \to 0} \frac{u(x_0+ry)}{r} \to u_{x_0}(y) \ \hbox{ locally uniformly.}\]
    where the blowup limit $u_{x_0}$ is 1-homogeneous. We also know $u_0$ is $C^{1,\beta}$ at $x_0 \in \partial \Omega_0 \setminus K_0$ and so, calling $\grad u_0(x_0) = s n_{x_0}$ where $s$ is the nonzero slope and $n_{x_0}$ is the inward normal to $\Omega_0$ at $x_0$,
    \[\lim_{r \to 0} \frac{u_0(x_0+ry)}{r} \to s(n_{x_0} \cdot y)_+\ \hbox{ locally uniformly.}\]
    Here $u_{x_0} \geq s(n_{x_0} \cdot y)_+$ will also minimize $\mathcal{E}[\{n_{x_0} \cdot y >0\}, \cdot]$ with respect to compact perturbations, since non-degeneracy can be used to show the $L^1_{loc}$ convergence of the indicator functions of the positivity sets. 
    
    Next we show that $u_{x_0}$ is linear, first in the interior half-space $\{y \cdot n_{x_0}>0\}$ and then in the (complementary) exterior half-space. By \cite{CaffarelliSalsa}*{Lemma 11.17}
    \[\lim_{r \to 0} \frac{u(x_0+ry)}{r} \to \alpha(n_{x_0} \cdot x)_+ \ \hbox{ locally uniformly in } \ x \cdot n_{x_0} \geq 0\]
    where $n_{x_0}$ is still the inward normal to $\Omega_0$ at $x_0$ and $\alpha^2 \leq 1 + \mu_+$.  This shows linearity in the interior half-space, but we do not have control in the exterior half-space yet.
    
    Suppose that $\Omega_- :=\{u_{x_0} >0\} \cap \{x \cdot n_{x_0} < 0\}$ is nontrivial then, by homogeneity, $0 \in \partial \Omega_-$.  But now $0$ is an outer regular point for $\Omega_-$, so we can apply \cite{CaffarelliSalsa}*{Lemma 11.17} to the blow-up of $u_{x_0}|_{\Omega_-}$ to find that
    \[\lim_{r \to 0} \frac{u_{x_0}(rz)}{r} \to \beta(n_{x_0} \cdot x)_- \ \hbox{ uniformly in } \ x \cdot n_{x_0} \leq 0\]
    for some $ 0 \leq \beta \leq 1+\mu_+$.  But then by homogeneity $u_{x_0} \equiv \alpha (x \cdot n_{x_0})_+ + \beta (x \cdot n_{x_0})_-$. This contradicts $\mathcal{E}[\{n_{x_0} \cdot y >0\}, \cdot]$ minimality unless $\beta =0$, since, when $\beta >0$, harmonic replacement decreases the Dirichlet energy without increasing the measure of the positivity set.
\end{proof}

The statement of \tref{regularity_estimate} is not quantitative so we also include a corollary explaining that the $C^{1,\beta}$ interior estimate can be quantified below $d_{sing}$.
\begin{corollary}\label{c.quantitative_regularity_estimate}
    Suppose that $d < d_{sing}$ and $u$, $u_0$ are as in the statement of \tref{regularity_estimate} in the domain $B_1$, with $K_0 = \emptyset$. Then 
    \[\|u\|_{C^{1,\beta}(\Omega(u) \cap B_{1/2})} \leq C(d,\beta,\|u_0\|_{C^{1,\beta}(\Omega(u_0) \cap B_1)},\|u\|_{L^\infty(\Omega(u) \cap B_1)}).\]
\end{corollary}
This can be proved by a compactness-contradiction argument for the initial flatness scale similar to \cite{FeldmanKimPozar}*{Lemma 3.9}.  If the first $\delta_0$-flat scale approached zero then, in the limit, we could find a minimizer which did not have a half-planar blow-up at $0$ contradicting \tref{regularity_estimate}.

We can now prove Corollary \ref{c.low_dim_regularity}, assuming the result of Theorem \ref{t.main}.

\begin{proof}[Proof of Corollary \ref{c.low_dim_regularity}]
    Any minimizing movements solution $u(t)$ satisfying the hypotheses of Theorem \ref{t.main} can be characterized by $u(t)\in \mathcal{M}[u(t_{i}), F(t)]$, where $t_i$ is the last monotonicity change before time $t$. Moreover, $u(t)$ is monotone on each $[t_i, t_{i+1}]$, so, for $d \leq d_{\text{sing}}$, we can apply the regularity estimate of \cref{quantitative_regularity_estimate}  $i+1$ times to reach time $t$ from time $0$, with global control over the $C^{1,\beta}$ regularity of $\Omega(u(t))$. Then, since $F(t)\in C^{1,\beta}(\partial U)$ uniformly in time, Schauder estimates imply the uniform in time regularity of $u(t)\in C^{1,\beta}(\overline{\Omega(u(t))})$.
\end{proof}

In higher dimensions we have the following regularity result with possible singular set.  

\begin{corollary}
    For any $ 0 < \beta < \frac{1}{2}$, if $F\in L^\infty([0, T]; C^{1,\beta}(\partial U))$, monotone, and $\Omega_0$ is $C^{1,\beta}$, except for a closed singular set $K_0 \subset \partial \Omega_0$. Then for each $t\in [0,T]$ there is a compact set $K(t)$ of dimension at most $d-d_{sing}$, disjoint from $\partial \Omega_0$, so that $u(t)\in C^{1,\beta}(\overline{\Omega(u(t))} \setminus (K_0 \cup K(t)))$.
\end{corollary}

The proof is similar to \cref{low_dim_regularity} proved above, using again our uniqueness result \tref{main} and now the full statement of \tref{regularity_estimate}.

\section{Minimizing movements for piecewise monotone forcing}

Our primary tool for understanding the minimizing movements will be the following max-min property of the energy functional:

\begin{lemma}\label{lem:energymaxmin}
For any $u_0,u_1,v_1$,
\begin{equation}\label{e.J_minmax}
    \mathcal{J}[u_1] + \mathcal{J}[v_1] = \mathcal{J}[u_1\wedge v_1] + \mathcal{J}[u_1\vee v_1]
\end{equation}
\begin{equation}\label{e.Diss_minmax}
    \Diss[u_0, u_1] + \Diss[u_0, v_1] = \Diss[u_0, u_1\wedge v_1] + \Diss[u_0, u_1\vee v_1]
\end{equation}
and thus
\begin{equation}\label{eq:same_ic_energy_minmax}
    \mathcal{E}[u_0, u_1] + \mathcal{E}[u_0, v_1] = \mathcal{E}[u_0, u_1\wedge v_1] + \mathcal{E}[u_0, u_1\vee v_1].
\end{equation}
Furthermore, for any $u_0, v_0, u_1, v_1$ with $\Omega(u_0) \subseteq \Omega(v_0)\cap \Omega(u_1)$ and $\Omega(v_0)\subseteq \Omega(v_1)$,
\begin{equation}\label{eq:ordered_ic_energy_minmax}
    \mathcal{E}[u_0, u_1] + \mathcal{E}[v_0, v_1] = \mathcal{E}[u_0, u_1\wedge v_1] + \mathcal{E}[v_0, u_1\vee v_1].
\end{equation}
\end{lemma}
\begin{proof}
    The properties \eqref{e.J_minmax}, \eqref{e.Diss_minmax}, and \eqref{eq:same_ic_energy_minmax} are standard.

    We check \eqref{eq:ordered_ic_energy_minmax}. By \eqref{e.J_minmax}, this reduces to checking the equality of dissipations. Using the ordering of positive sets, we have
    \begin{align*}
        &\frac{1}{\mu_+}\left(\Diss[u_0, u_1] + \Diss[v_0, v_1] - \Diss[u_0, u_1\wedge v_1] - \Diss[v_0, u_1\vee v_1]\right) \\&= |\Omega(u_1)\setminus \Omega(u_0)| + |\Omega(v_1)\setminus \Omega(v_0)| \\&\qquad\cdots- |\Omega(u_1\wedge v_1)\setminus \Omega(u_0)| - |\Omega(u_1\vee v_1)\setminus \Omega(v_0)|
        \\&= |\Omega(u_1)| -  |\Omega(u_0)| + |\Omega(v_1)| - |\Omega(v_0)| - |\Omega(u_1\wedge v_1)| \\
        &  \qquad \cdots+ |\Omega(u_0)| - |\Omega(u_1\vee v_1)| + |\Omega(v_0)|
        \\&= |\Omega(u_1)| + |\Omega(v_1)| - |\Omega(u_1\wedge v_1)| - |\Omega(u_1\vee v_1)|
        \\&= 0
    \end{align*}
\end{proof}

The typical application of this lemma starts with $u_1, v_1$ as known energy minimizers, and concludes that $u_1\vee v_1$ and $u_1\wedge v_1$ must also be minimizers. Of course, this requires some additional assumption on the forcing data of $u_1, v_1$, but we have flexibility as to what assumptions we might make.

Recall the notation $\mathcal{M}[u,F]$ defined in \eqref{eq:argmin_notation} for the set of minimizers with forcing $F$ from profile $u$. As a first application of the max-min property, we show a type of comparison holds for candidate minimizing movements.

\begin{lemma}\label{lem:minmaxlim}
    $\mathcal{M}[u, F]$ is closed under pointwise maximums and minimums, as well as pointwise limits.    In particular there exist maximal and minimal elements of $\mathcal{M}[u,F]$ which we denote, respectively, $\mathcal{M}^{\max}[u, F]$ and $\mathcal{M}^{\min}[u, F]$.
\end{lemma}
\begin{proof}
    The claim for pointwise maximums and minimums follows immediately from \eqref{eq:same_ic_energy_minmax}, since if $u_1, v_1$ are minimizers on the left-hand-side, then $u_1\wedge v_1, u_1\vee v_1$ must each also minimize $\mathcal{E}[u_0, \cdot]$ for that equality to hold.

    Next, let $v_n\in \mathcal{M}[u,F]$ be a sequence which converges pointwise to some $v$. The $v_n$ all have the same value for $\mathcal{E}[u, v_n]$, so the Dirichlet energies $\mathcal{D}[v_n]$ are uniformly bounded and the $v_n$ are weakly compact in $H^1_{\mathrm{loc}}$. Using the uniform Lipschitz continuity and nondegeneracy of the $v_n$ near the free boundary from Lemma \ref{lem:basic_estimates}, we get that $\one_{\Omega(v_n)}$ converges pointwise to $\one_{\Omega(v)}$. We conclude that $v\in H^1_0(U) + F$, and dominated convergence gives convergence for all of the measure terms in order to conclude that $|\Omega(v)| + \Diss[u, v] = \lim_n |\Omega(v_n)| + \Diss[u, v_n]$. On the other hand, $\mathcal{D}[v] \leq \liminf_n \mathcal{D}[v_n]$, so $\mathcal{E}[u, v] \leq \liminf \mathcal{E}[u, v_n]$ and $v\in \mathcal{M}[u, v]$.

    The existence of the maximal and minimal elements, $\mathcal{M}^{\max}[u, F]$ and $\mathcal{M}^{\min}[u, F]$, is now straightforward.  For example, among minimizers $|\Omega(v)|$ is bounded above by the energy $\mathcal{E}[u, v]$. Thus, $\{ |\Omega(v)| : v\in \mathcal{M}[u, F] \}$ has a supremum and an infimum. We can then use pointwise maximums to construct an increasing sequence $v_n$ with $|\Omega(v_n)|$ converging upwards to the supremal volume. The maximum volume profile must be ordered with respect to all other profiles since it does not gain volume under pointwise maximum, and a similar argument holds for the minimum volume profile.
\end{proof}

In light of Lemma \ref{lem:energymaxmin}, the maximum and minimum minimizers are especially useful to work with. In particular they are monotone with respect to monotone changes of the boundary condition.

\begin{lemma}\label{lem:forcing_comparison}
    If $F\leq G$, then $\mathcal{M}^{\min}[u_0, F] \leq \mathcal{M}^{\min}[u_0, G]$ and $\mathcal{M}^{\max}[u_0, F] \leq \mathcal{M}^{\max}[u_0, G]$.

    If, additionally, $F < G$ somewhere on each component of $\partial U$, then $\mathcal{M}^{\max}[u_0, F] \leq \mathcal{M}^{\min}[u_0, G]$.
\end{lemma}
\begin{proof}
    For any $u\in \mathcal{M}[u_0, F]$ and $v\in \mathcal{M}[u_0, G]$, we have that $u \wedge v\in \mathcal{M}[u_0, F]$ and $u\vee v\in \mathcal{M}[u_0, G]$. This is because \eqref{eq:same_ic_energy_minmax} gives
    \[ \mathcal{E}[u_0, u] + \mathcal{E}[u_0, v] = \mathcal{E}[u_0, u\wedge v] + \mathcal{E}[u_0, u\vee v] \]
    Since $F\leq G$, we have $u\wedge v = F$ and $u\vee v = G$ on $\partial U$. Hence, minimality of the energies on the left side demands minimality of the energies on the right side, so $u\wedge v$ and $u\vee v$ are also minimizers. Maximality of $\mathcal{M}^{\max}[u_0, G]$ and minimality of $\mathcal{M}^{\min}[u_0, F]$ then lead to the first inequalities. Specifically $\mathcal{M}^{\min}[u_0, F] \wedge \mathcal{M}^{\min}[u_0, G] \in \mathcal{M}[u_0,F]$ and so minimality implies
    \[\mathcal{M}^{\min}[u_0, F] \leq \mathcal{M}^{\min}[u_0, F] \wedge \mathcal{M}^{\min}[u_0, G]\]
    implying that $\mathcal{M}^{\min}[u_0, F]\leq\mathcal{M}^{\min}[u_0, G]$ and similar for the other claimed inequality.
    
    Applying the first sentence of the proof again, $\mathcal{M}^{\max}[u_0, F]\wedge \mathcal{M}^{\min}[u_0, G] \in \mathcal{M}[u_0, F]$ and is therefore harmonic in its support. Now, assuming that $F<G$ somewhere on each component of $\partial U$, we have that $\mathcal{M}^{\max}[u_0, F]\wedge \mathcal{M}^{\min}[u_0, G] - \mathcal{M}^{\max}[u_0, F]$ is a difference of harmonic functions which is identically zero in neighborhoods on each of its components, so the strong maximum principle implies that $\mathcal{M}^{\max}[u_0, F]\wedge \mathcal{M}^{\min}[u_0, G] = \mathcal{M}^{\max}[u_0, F]$.
\end{proof}

The key structural property of the dissipation distance that we used in \eqref{eq:ordered_ic_energy_minmax} and the subsequent results is its additivity for monotonically ordered profiles: if $\Omega_0 \subset \Omega_1 \subset \Omega_2$, then $\Diss[\Omega_0, \Omega_1] + \Diss[\Omega_1, \Omega_2] = \Diss[\Omega_0, \Omega_2]$. Under certain circumstances, we can also use this additivity to argue that a minimizer relative to one profile is also a minimizer relative to another profile. This will be the first step toward proving uniqueness results for the evolution.

\begin{lemma}\label{l.diff_ics_merge}
    Let $u_0, v_0, F$ be such that
    \[ \Omega(u_0)\cup \Omega(v_0) \subseteq \Omega(\mathcal{M}^{\max}[u_0, F])\cap \Omega(\mathcal{M}^{\max}[v_0, F]) \]
    Then
    \[ \{ u' \in \mathcal{M}[u_0, F] : \Omega(u_0)\cup \Omega(v_0) \subseteq \Omega(u') \} = \{ u' \in \mathcal{M}[v_0, F] : \Omega(u_0)\cup \Omega(v_0) \subseteq \Omega(u') \} \]
    In particular, $\mathcal{M}^{\max}[u_0, F] = \mathcal{M}^{\max}[v_0, F]$.
    
    Similarly, if
    \[ \Omega(\mathcal{M}^{\min}[u_0, F])\cup \Omega(\mathcal{M}^{\min}[v_0, F]) \subseteq \Omega(u_0)\cap \Omega(v_0) \]
    Then $\mathcal{M}^{\min}[u_0, F] = \mathcal{M}^{\min}[v_0, F]$, and
    \[ \{ u' \in \mathcal{M}[u_0, F] : \Omega(u_0)\cap \Omega(v_0) \supseteq \Omega(u') \} = \{ u' \in \mathcal{M}[v_0, F] : \Omega(u_0)\cap \Omega(v_0) \supseteq \Omega(u') \} \]
    
\end{lemma}
\begin{proof}
    The assumption implies that both energy minimization problems have minimizers $u'$ which satisfy $\Omega(u_0)\cup \Omega(v_0) \subseteq \Omega(u')$. For such profiles, we have
    \begin{gather*}
       \Diss(u_0, u') = \mu_+|\Omega(u') \setminus \Omega(u_0)|\\
       \Diss(v_0, u') = \mu_+|\Omega(u') \setminus \Omega(v_0)|
    \end{gather*}
    leading to
    \begin{align*}
        \mathcal{E}[u_0, u'] &= \mathcal{E}[v_0, u'] + \mu_+(|\Omega(u')\setminus \Omega(u_0)| - |\Omega(u')\setminus \Omega(v_0)|)
        \\&= \mathcal{E}[v_0, u'] + \mu_+(|\Omega(v_0)\setminus \Omega(u_0)| - |\Omega(u_0)\setminus \Omega(v_0)|)
    \end{align*}
    Thus, the energy functionals differ by a constant in this case. By considering the profiles satisfying $\Omega(u_0)\cup \Omega(v_0) \subseteq \Omega(u')$ that minimize energy for $u_0$ and $v_0$, we see that the minimum energies must also differ by the same constant. It follows that any such minimizer is a minimizer for both $u_0$ and $v_0$.

    The decreasing case is similar; for $u'$ with $\Omega(u') \subset \Omega(u_0)\cap \Omega(v_0)$, one has
    \begin{align*}
        \mathcal{E}[u_0, u'] &= \mathcal{E}[v_0, u'] + \mu_-(|\Omega(u_0)\setminus \Omega(u')| - |\Omega(v_0)\setminus \Omega(u')|)
        \\&= \mathcal{E}[v_0, u'] + \mu_-(|\Omega(u_0)\setminus \Omega(v_0)| - |\Omega(v_0)\setminus \Omega(u_0)|).
    \end{align*}
\end{proof}

We are now ready to derive the main lemmas of this section, which fully characterize minimizing movements under monotone forcing.  Our first result shows that, by appropriate choice of the minimizer at each step, the time incremental scheme for monotone forcing is stable under partition refinement. For a sequence of nested partitions, the time incremental scheme with these choices coincides with its continuous time limit at all partition times.

\begin{lemma}\label{lem:max_min_mov_stability}
    If $F_0 \leq F_1 \leq F_2$, and $u_0 \in \mathcal{M}[u_0, F_0]$, then
    \[ \mathcal{M}^{\max}[\mathcal{M}^{\max}[u_0, F_1], F_2] = \mathcal{M}^{\max}[u_0, F_2]. \]
    In other words, for increasing forcing, a time incremental procedure which chooses the largest energy minimizer in each step is stable under partition refinement.

    Similarly, if $F_0 \geq F_1 \geq F_2$, then
    \[ \mathcal{M}^{\min}[\mathcal{M}^{\min}[u_0, F_1], F_2] = \mathcal{M}^{\min}[u_0, F_2] \]
    in other words, for decreasing forcing, a time incremental procedure which chooses the smallest energy minimizer in each step is stable under partition refinement.
\end{lemma}

\begin{proof}
    By \lref{diff_ics_merge}, it suffices to show that
    \[ \Omega(u_0) \cup \Omega(\mathcal{M}^{\max}[u_0, F_1]) \subseteq 
 \Omega(\mathcal{M}^{\max}[\mathcal{M}^{\max}[u_0, F_1], F_2]) \cap \Omega(\mathcal{M}^{\max}[u_0, F_2]) \]
    Most of the inclusions come for free by \eqref{eq:same_ic_energy_minmax}, which gives
    \[ \Omega(u_0) \subseteq \Omega(\mathcal{M}^{\max}[u_0, F_1]) \subseteq \Omega(\mathcal{M}^{\max}[\mathcal{M}^{\max}[u_0, F_1], F_2]) \]
    Lemma \ref{lem:forcing_comparison} gives that $\mathcal{M}^{\max}[u_0, F_1] \leq \mathcal{M}^{\max}[u_0, F_2]$, so we conclude. The decreasing case is identical.
\end{proof}

As a consequence of the previous lemma, if $F$ is monotone increasing then
\begin{equation}
    u(t) = \mathcal{M}^{\max}[u_0, F(t)]
\end{equation}
is a minimizing movements solution and is the pointwise maximum of all minimizing movements solutions. Thus, our efforts now turn to showing that any other choice of minimizing movements yields effectively the same solution. In fact, we find that for strictly increasing forcing, a nonmaximal minimizing movements procedure catches up to the maximal minimizing movements in each subsequent step.

\begin{lemma}\label{lem:comparison_to_maximal_min_mov}
    If $F_0 < F_1 < F_2$ on $\partial U$ and $u_0\in \mathcal{M}[u_0, F_0]$, then for any $u_1\in \mathcal{M}[u_0, F_1], u_2\in \mathcal{M}[u_1, F_2]$, we have
    \[ \mathcal{M}^{\max}[u_0, F_1] \leq u_2 \]
    The similar result holds in the decreasing case.
\end{lemma}
\begin{proof}
    By Lemma \ref{lem:forcing_comparison}, we have $u_0 \leq u_1 \leq u_2$, and we also have $u_0 \leq \mathcal{M}^{\max}[u_0, F_1]$. Then applying \eqref{eq:ordered_ic_energy_minmax}, we get that
    \[ \mathcal{E}[u_0, \mathcal{M}^{\max}[u_0, F_1]] + \mathcal{E}[u_1, u_2] = \mathcal{E}[u_0, \mathcal{M}^{\max}[u_0, F_1]\wedge u_2] + \mathcal{E}[u_1, \mathcal{M}^{\max}[u_0, F_1]\vee u_2] \]
    Then $\mathcal{M}^{\max}[u_0, F_1]\vee u_2\in \mathcal{M}[u_1, F_2]$, so a strong maximum principle argument using that $F_1 < F_2$ yields that $\mathcal{M}^{\max}[u_0, F_1] \leq u_2$.
\end{proof}

The previous two lemmas allow us to trap a general discrete scheme between time shifts of the maximum minimizing movements solution. With this, we can show that any choice of discrete scheme leads to the same continuous time limit at all times where the maximum minimizing movements solution does not jump. At jumps, any intermediate globally stable state between the left and right limits may be taken by the continuous time solution.

\begin{theorem}\label{thm:strictly_monotone_uniqueness}
    Suppose $F$ is strictly increasing on $[0, T]$, and $u_0\in \mathcal{M}[u_0, F(0)]$. For $t > 0$, define
\begin{gather*}
    u^{\min}(t) = \mathcal{M}^{\min}[u_0, F(t)] \\
    u^{\max}(t) = \mathcal{M}^{\max}[u_0, F(t)]
\end{gather*}
    Then for any partition $P = \{ t_1, \dots, t_n \}$, we have that $\{ u^{\min}(t_k) \}$ and $\{ u^{\max}(t_k) \}$ are minimizing movements for the partition. Any general minimizing movements $u^k$ satisfy $u^k\in \mathcal{M}[u_0, F(t_k)]$ for each $k$, and thus $u^{\min}(t_k) \leq u^k \leq u^{\max}(t_k)$.

    Moreover, $u^{\min}(t)$ and $u^{\max}(t)$ have the same left and right limits for all $t$. Hence, we conclude that all minimizing movements solutions for $F$ jump at the same set of times and agree away from those times, so we may say that the minimizing movements solution is unique up to replacement at jumps.
\end{theorem}
\begin{proof}
    That $u^{\max}$ yields a valid discrete scheme follows immediately from stability under refinement, Lemma \ref{lem:max_min_mov_stability}.
    
    To see that $u^{\min}$ yields a valid discrete scheme, we use Lemma \ref{lem:forcing_comparison} and the strict monotonicity of $F$ to get that $u^{\min}(t)$ is increasing in time. Then we note that
    \[ \Omega(u_0) \subseteq \Omega(u^{\min}(t_{k-1})) \subseteq \Omega(u^{\min}(t_k)) \cap \Omega(\mathcal{M}^{\max}[u^{\min}(t_{k-1}), F(t_k)]) \]
    where we use monotonicity of $u^{\min}$ as well as monotonicity of the movement from $u^{\min}(t_{k-1})$ since $F(t_{k-1}) < F(t_k)$. We can then apply Lemma \ref{l.diff_ics_merge} to conclude that the set of minimizers with respect to $u_0$ and $u^{\min}(t_{k-1})$ are the same, and thus $u^{\min}(t_k) = \mathcal{M}^{\min}[u^{\min}(t_{k-1}), F(t_k)]$.

    Having shown that both $u^{\min}$ and $u^{\max}$ are stable under refinement, we can return to Lemma \ref{lem:comparison_to_maximal_min_mov} and take the limit.
    \[ u^{\max}(t-) \leq u^{\min}(t) \leq u^{\max}(t) \]
    for all $t > 0$. This implies that $u^{\min}, u^{\max}$ have the same left and right limits.

    Finally, we check that a generic discrete scheme $u^k$ is contained between $u^{\min}$ and $u^{\max}$.  This holds at the first step, and inductively, if $u^{\min}(t_k) \leq u^k \leq u^{\max}(t_k)$, then \eqref{eq:ordered_ic_energy_minmax} implies that
    \[ \mathcal{M}^{\min}[u^{\min}(t_k), F(t_{k+1})] \leq u^{k+1} \leq \mathcal{M}^{\max}[u^{\max}(t_k), F(t_{k+1})] \]
    for any $u^{k+1}\in \mathcal{M}[u^k, F(t_{k+1})]$. Since we have already verified that $u^{\min}(t_{k+1}) = \mathcal{M}^{\min}[u^{\min}(t_k), F(t_{k+1})]$ and $\mathcal{M}^{\max}[u^{\max}(t_k), F(t_{k+1})] = u^{\max}(t_{k+1})$, we conclude in particular that
    \[ u_0 \leq u^k \leq u^{k+1} \leq u^{\max}(t_{k+1}) \]
    Lemma \ref{l.diff_ics_merge} then implies that $u^{k+1} \in \mathcal{M}[u_0, F(t_{k+1})]$.
\end{proof}

\subsection{Piecewise monotone forcing}

Having completed our analysis of monotone forcing, we now turn toward the piecewise monotone setting. The main complication this adds is that at the discrete level, the step across the monotonicity change has no ordering. As a result, unlike in the monotone case, we cannot hope for the piecewise monotone discrete scheme to coincide with its continuous time limit; there can be some error appearing beyond the first monotonicity change.

However, we do not completely lose control, since this nonmonotone step cannot travel too far energetically as the fineness of the partition tends to 0. The arguments in this section show that when we pass to the continuous time limit, we regain a characterization of the minimizing movements solution at each time in terms of finitely many minimization problems.

\begin{lemma}\label{lem:convergence_at_monotonicity_change}
Let $s_k, t_k$ be a sequence of times such that $s_k$ converges from above to a monotonicity change $s_0$ of $F$, and $t_k > s_k$ is contained in the subsequent monotone interval and converges to some $t$. Let $\Omega_k$ be a sequence of sets converging in $L^1(U)$ to some $\Omega_0$ which satisfies global stability for $F_0$.

Then for any $u_k \in \mathcal{M}[\Omega_k, F(s_k)]$ and $v_k\in \mathcal{M}[u_k, F(t_k)]$, and any subsequential limit $v_0$ of the $v_k$, we have $v_0\in \mathcal{M}[\Omega_0, F(t)]$.
\end{lemma}
\begin{proof}
    We have already seen the main ingredients for this type of argument: by weak lower semicontinuity of the $H^1$ seminorm and Fatou's lemma applied to the measure terms, any subsequential limit of profiles relative to fixed initial data can only lose energy. Hence, by competitor arguments, any limit of minimizers for initial data converging in $L^1(U)$ and forcing converging in $H^{1/2}(\partial U)$ is a minimizer. In particular, we can apply this principle to $u_k$ in the statement of the lemma to get that any subsequential limit $u_0$ is in $\mathcal{M}[\Omega_0, F(s_0)]$.

    Thus, passing to convergent subsequences, we would like to show that $v_0\in \mathcal{M}[u_0, F(t_0)]$. Supposing otherwise, then we can take a profile in $\mathcal{M}[u_0, F(t_0)]$ with strictly less energy then $v_0$. Then, choose $k$ sufficiently large to ensure that $|\Omega_k \Delta \Omega_0|$ and $\|F(t_k) - F(t_0)\|_{H^{1/2}(\partial U)}$ are both much smaller than this energy gap, so that we can modify our profile into a competitor for $v_k$ relative to $\Omega_k$. This yields a contradiction, so we conclude.
    
\end{proof}

\begin{theorem}\label{thm:piecewise_monotone_1}
    Suppose $0 = t_0 < \dots < t_N = T$ are such that $F$ is strictly monotone on each $(t_i, t_{i+1})$, and $\Omega_0\in \mathcal{M}[\Omega_0, F(t_0)]$. If $u$ is a minimizing movements solution with forcing $F$ and initial data $\Omega_0$, then we have $u(t) \in \mathcal{M}[u(t_i), F(t)]$ for each $i$ and each $t\in [t_i, t_{i+1}]$.
\end{theorem}
\begin{proof}
    Note that the claim holds in $[t_0, t_1]$ by applying Theorem \ref{thm:strictly_monotone_uniqueness} to the monotone minimizing movements. We can additionally note that if $t_\delta$ is the last time step in $[t_0, t_1]$ of a $\delta$-fine discrete scheme, then the profile at that step converges to a profile in $\mathcal{M}[u(t_0), F(t_1)]$ as $\delta\to 0$.

    Let us assume that this second property holds inductively; that the sequence of last discrete steps before $t_i$ converges to a minimizer with forcing $F(t_i)$ relative to initial data $\Omega(u(t_{i-1}))$. Then, looking into the next monotone interval, we apply the lemma, with $\Omega_k$ being the profile of this last discrete step, and $s_k$ being the time of the first step into the next monotone interval. Then the conclusion of Lemma \ref{lem:convergence_at_monotonicity_change} proves the theorem on $[t_i, t_{i+1}]$, and gives us the inductive step via the convergence of the next sequence of last discrete steps.
    
\end{proof}

\begin{proposition}
Suppose $0 = t_0 < \dots < t_N = T$ are such that $F$ is strictly monotone on each $(t_i, t_{i+1})$, and $\Omega_0\in \mathcal{M}[\Omega_0, F(t_0)]$. Any $u(t)$ chosen by
\[ \begin{cases}
u(t)\in \mathcal{M}[u(t_i-), F(t)] & t_i < t \leq t_{i+1}
\end{cases} \]
is a minimizing movements solution. In light of Theorem \ref{thm:piecewise_monotone_1}, this characterizes all minimizing movement solutions.
\end{proposition}
\begin{proof}
We construct a system of partitions such that the monotonicity changes appear in every partition, and any time where $u(t)$ jumps within a monotone interval appears in all but finitely many partitions (since we do not rule out the possibility of countably many jumps). Additional times may be added arbitrarily to obtain $\delta$-fine partitions.

Now, we can choose the values of the discrete scheme to exactly coincide with the desired limit $u(t)$. This is relying on the stability under refinement property of the monotone minimizing movements given by Lemma \ref{lem:max_min_mov_stability}, which ensures that if $u(t)\in \mathcal{M}[u(t_i), F(t)]$ for $t\in [t_i, t_{i+1}]$, then also $u(t)$ can be obtained from any minimizing discrete scheme with strictly monotone forcing starting from $u(t_i)$ and ending at $F(t)$. In particular, since we have the monotonicity change times in all partitions, we can inductively apply the lemma to ensure the value at the next monotonicity change is chosen accordingly and the next monotone interval will be handled correctly.

Now, given such a discrete scheme, it is clear that at jumps and monotonicity changes for $u(t)$, the discrete scheme is eventually constant as $\delta\to 0$, so $u(t)$ is the limit of the discrete scheme. On the other hand, from Theorem \ref{thm:strictly_monotone_uniqueness} the jump times of the limit of the discrete scheme on $[t_i, t_{i+1})$ are determined by $u(t_i)$, and the limit is the unique element of $\mathcal{M}[u(t_i), F(t)]$ for non-jump times, so we have shown that the discrete scheme converges to $u(t)$.
\end{proof}

\begin{remark}
    It is not clear whether all of the possible minimizing movements solutions can be created from partitions avoiding monotonicity changes. A somewhat related question is whether the situation of more than two elements in $\mathcal{M}[\Omega, F]$ can occur nontrivially, or if it is always the result of independent merging events occuring at the same time.
\end{remark}

Finally, though we have completely characterized minimizing movements solutions for the piecewise monotone case, it is reasonable to ask whether the different solutions found represent genuinely different branches. We answer this affirmatively by showing that after jumping at a monotonicity change, the solution stays near the jump profile. On the other hand, in view of Proposition \ref{l.diff_ics_merge}, if there is enough change in forcing, different branches may merge again at a later time.

\begin{proposition}\label{prop:value_at_monotonicity_change}
Let $u$ be as in the previous statement. Then $u$ is right continuous at each monotonicity change.
\end{proposition}
\begin{proof}
Let us restrict without loss of generality to the case where $t_i$ is a local maximum. We have $u(t_i)\in \mathcal{M}[u(t_i-), F(t_i)]$ and $u(t_i+)\in \mathcal{M}[u(t_i), F(t_i)]\cap \mathcal{M}[u(t_i-), F(t_i)]$. Also, as a property of energy solutions, we know that $u(t_i)$ is between $u(t_i-)$ and $u(t_i+)$ on each connected component.

By Theorem \ref{thm:piecewise_monotone_1}, $u(t)$ minimizes energy with respect to the profile of the previous monotonicity change. Combining this with the comparison results of Lemma \ref{lem:forcing_comparison}, we get that $u(t) \leq \mathcal{M}^{\min}[u(t), F(t_i)]$ for $t\in [t_{i-1}, t_i]$ and $u(t) \leq \mathcal{M}^{\min}[u(t_i), F(t_i)]$ for $t\in [t_i, t_{i+1}]$. Passing to the limit, we must have $u(t_i-) = \mathcal{M}^{\min}[u(t_i-), F(t_i)] \leq u(t_i)$ and $u(t_i+) = \mathcal{M}^{\min}[u(t_i), F(t_i)] \leq u(t_i)$. Note that the dissipation triangle inequality implies that for any $u'$, we have $\Diss(u(t_i-), u') \leq \Diss(u(t_i-), u(t_i)) + \Diss(u(t_i), u')$; in particular, since $u(t_i)\in \mathcal{M}[u(t_i-), F(t_i)]$, any minimizer in $\mathcal{M}[u(t_i), F(t_i)]$ is a minimizer in $\mathcal{M}[u(t_i-), F(t_i)]$. It follows directly that 
\[\mathcal{M}^{\min}[u(t_i-), F(t_i)] \leq \mathcal{M}^{\min}[u(t_i), F(t_i)],\]
and $u(t_i-) \leq u(t_i+)$. Then the ordering per component implies that $u(t_i-) \leq u(t_i) \leq u(t_i+)$, so $u(t_i) = u(t_i+)$.

\end{proof}

\appendix

\section{Regularity of energy solutions with \texorpdfstring{$x$}{x}-dependent forcing}\label{a.energy-solution-properties}
In this appendix we establish the basic theory of the Dirichlet-driven rate-independent motion when the Dirichlet forcing has $x$-dependence.
\begin{proposition}\label{p.energy-solutions-temporal-reg}
    Let $u$ be an energy solution on $[0, T]$ with forcing $F(x, t)$ uniformly bounded away from 0. Then \[\one_{\Omega(u(t))}\in \mathrm{BV}([0, T]; L^1(\R^d))\]
    and
    \[ \mathcal{J}[u(t)] \in \mathrm{BV}([0,T]; \R). \]
\end{proposition}

We will establish the proposition in two steps which we will write as Lemmas.  First we show the $\textup{BV}_t$ estimate of the positivity sets $\Omega(u(t))$ via a typical Gr\"onwall argument with the energy dissipation inequality, see \lref{set-time-BV-reg}.  Then we show the $\textup{BV}_t$ estimate of the energy using global stability in \lref{energy-time-BV-reg}.  The $x$-dependence of $F(t,x)$ presents some additional technical difficulties over the case of $F(t)$ considered in \cite{FeldmanKimPozar}*{Lemma 5.7}.

\begin{lemma}\label{l.set-time-BV-reg}
    Suppose $u$ is an energy solution on $[0, T]$. Then $\one_{\Omega(u(t))}$ is in $ \textup{BV}([0, T]; L^1)$ with the estimate
    \begin{equation}
        [\one_{\Omega(u(t))}]_{\textup{BV}([0,T], L^1)} \leq \frac{1}{\mu_+\wedge \mu_-}\mathcal{J}(u(0))\exp(\|\partial_t (\log F)_+\|_{L^1([0, T]; L^\infty(\partial U))})
    \end{equation}
\end{lemma}
\begin{proof}
Define $Q(t) := \langle 2\dot{F}(t), \frac{\partial u}{\partial n}\rangle_{L^2(\partial U)}$ and define the total dissipation
\[\Diss(\Omega(u(t));[0,T]):= \sup_{\mathcal{P}} \sum_{j} \Diss(\Omega(u(t_j)),\Omega(u(t_{j+1}))).\]  
The supremum is over $\mathcal{P}$ partitions of $[0,T]$. By the energy dissipation inequality, applied repeatedly on an arbitrary partition,
\[ \mathcal{J}(u_0)-\mathcal{J}(u(t))  + \int_{0}^{t} \mathcal{Q}(s) \ ds \geq \Diss(\Omega(u(t));[0,T]) \geq (\mu_+ \wedge \mu_-) [{\bf 1}_{\Omega(\cdot)}]_{\textup{BV}([0,t];L^1)}.\]
Note by integration by parts
\[\mathcal{D}(u(t))  = \int_{\partial U} F(t,x) \frac{\partial u}{\partial n}(t,x) \ dS.\]
We can use this to write, using the positivity of the integrand $F \frac{\partial u}{\partial n}$ on $\partial U$,
\begin{align*}
\mathcal{Q}(t) &= \langle 2 \dot{F}(t,\cdot), \frac{\partial u}{\partial n}(t,\cdot)\rangle_{L^2(\partial U)} \\
&= \langle 2 \frac{d}{dt}(\log F)(t,\cdot), F(t,\cdot)\frac{\partial u}{\partial n}(t,\cdot)\rangle_{L^2(\partial U)}\\
&\leq  2\| [\frac{d}{dt}(\log F)(t,\cdot)]_+\|_{L^\infty(\partial U)}\mathcal{D}(t).
\end{align*}
Calling $g(t) = \| [\frac{d}{dt}(\log F)(t,\cdot)]_+\|_{L^\infty}$ and $G(t) = \exp(\int_0^tg(s) ds)$ we obtain, 
\[\mathcal{Q}(t) \leq 2 \frac{d}{dt}(\log G)(t)\mathcal{D}(t).\]
Also it is worth recording, for any $t_1 \geq t_0 \geq 0$,
\[  \frac{F(t_1,x)}{F(t_0,x)} = \exp\left(\int_{t_0}^{t_1} \frac{d}{dt}(\log F)(t,x) \ dt\right)\]
and so, bounding the integrand above and below by $\pm g(t)$ in the natural way,
\begin{equation}\label{e.F-G-bounds}
\frac{G(t_0)}{G(t_1)} \leq \frac{F(t_1,x)}{F(t_0,x)} \leq \frac{G(t_1)}{G(t_0)} \ \hbox{ for all } \ x.
\end{equation}

Note that $\mathcal{D}(t) \leq \mathcal{J}(u(t))$ so in particular
\begin{align*}
\mathcal{D}(t) &\leq \mathcal{J}(u_0)+\int_{0}^{t}\mathcal{Q}(s) \ ds \\
&\leq \mathcal{J}(u_0)+\int_{0}^{t}2 \frac{d}{ds}(\log G)(s) \mathcal{D}(s) \ ds
\end{align*}
for all $0 \leq t \leq T$ so by Gr\"onwall
\[\mathcal{D}(t) \leq \mathcal{J}(u_0) G(t)^2.\]

Thus for any $t_1 \geq t_0 \geq 0$
\begin{align*}
\mathcal{J}(u(t_1)) + \mu_+ \wedge \mu_- [{\bf 1}_{\Omega(\cdot)}]_{\textup{BV}([t_0,t_1];L^1)} &\leq \mathcal{J}(u(t_0))\left(1+\frac{1}{G(t_0)^2}\int_{t_0}^{t_1}2\dot{G}(s)G(s) \ ds\right) \\
&= \mathcal{J}(u(t_0))\left(1+\frac{G(t_1)^2-G(t_0)^2}{G(t_0)^2}\right)\\
&=\mathcal{J}(u(t_0))\frac{G(t_1)^2}{G(t_0)^2}
\end{align*}
or, to summarize,
\begin{equation}\label{e.energy-ineq-BV}
\mathcal{J}(u(t_1)) + \mu_+ \wedge \mu_- [{\bf 1}_{\Omega(\cdot)}]_{\textup{BV}([t_0,t_1];L^1)}\leq \mathcal{J}(u(t_0))\frac{G(t_1)^2}{G(t_0)^2}.
\end{equation}
In particular
\[[{\bf 1}_{\Omega(\cdot)}]_{\textup{BV}([0,t];L^1)} \leq \frac{1}{\mu_+\wedge \mu_-} \mathcal{J}(u_0)\frac{G(t)^2}{G(0)^2}.\]
\end{proof}

\begin{lemma}\label{l.energy-time-BV-reg}
    Suppose $u$ is an energy solution on $[0, T]$. Then 
    \begin{equation*}
        |\mathcal{D}(u(t_1)) - \mathcal{D}(u(t_0))| \leq C\|F(t_1) - F(t_0)\|_{H^{1/2}(\partial U)} + (1 + \mu_+ \vee \mu_-)|\Omega(u(t_1))\Delta \Omega(u(t_0))|.
    \end{equation*}
     Thus, $\mathcal{D}(u(t)), \mathcal{J}(u(t))\in \textup{BV}([0, T]; \R)$.
\end{lemma}
Note that \pref{energy-solutions-temporal-reg} follows directly from \lref{set-time-BV-reg} and the final conclusion of \lref{energy-time-BV-reg}.
\begin{proof}

Using the standing assumption that $F$ is bounded away from $0$ and Lemma \ref{lem:domain_size_estimate}, we fix an $\varepsilon_1 > 0$ and a domain $\omega\subset U$ such that any globally stable solution for $F(t)$ for any $t$ is at least $\varepsilon_1$ on $\omega$. Then, take $h^{s,t}\in H^1(U)$ to solve
\begin{equation}\label{eq:additive_boundary_data_corrector}
    \begin{cases}
    \Delta h^{s,t} = 0 \hbox{ in }\omega \\
    h^{s,t} = F(t) - F(s) \hbox{ on }\partial U \\
    h^{s,t} = 0 \hbox{ outside }\omega
    \end{cases}
\end{equation}
Note that $\|h^{s,t}\|_{H^1(U)} \leq C\|F(t) - F(s)\|_{H^{1/2}(\partial U)}$, with constant depending on the fixed domain $\omega$. On the other hand, by maximum principle, $\|h^{s,t}\| \leq \|F(t) - F(s)\|_{L^\infty(\partial U)}$, so assuming $\|F(t) - F(s)\|_{L^\infty(\partial U)} \leq \varepsilon_1$, the perturbation by $h^{s,t}$ does not modify the positivity set of a globally stable profile. We will assume that $|t_1 - t_0|$ is sufficiently small based on the uniform continuity of $F$ in $L^\infty(\partial U)$ for this condition to hold. Now, $u(t_0) + h^{t_0, t_1}\in F(t_1) + H^1_0(U)$, so global stability implies that
\[ \mathcal{J}[u(t_1)] \leq \mathcal{J}[u(t_0) + h^{t_0, t_1}] + \Diss[u(t_1), u(t_0)] \]
We can expand the energy terms and regroup as
\[ \mathcal{D}[t_1] - \mathcal{D}[t_0] \leq 2\langle \nabla u(t_0), h^{t_0, t_1}\rangle_{L^2(U)} + \mathcal{D}[h^{t_0, t_1}] + (1 + \mu_+ + \mu_-)|\Omega(u(t_0))\Delta \Omega(u(t_1))| \]
Using the energy bound from Lemma \ref{lem:domain_size_estimate}, we have a bound for $\|\nabla u\|_{L^\infty_t L^2_x}$ in terms of $\|F\|_{L^\infty_t H^{1/2}_x}$. Thus, we can combine this with the bounds for $h$ in terms of $F$ to reduce to
\begin{equation}
\mathcal{D}[t_1] - \mathcal{D}[t_0] \leq C\|F(t_1) - F(t_0)\|_{H^{1/2}(\partial U)} + (1 + \mu_+ + \mu_-)|\Omega(u(t_0))\Delta \Omega(u(t_1))|    
\end{equation}
Summing this for sufficiently fine partitions, we conclude
\begin{equation}
[\mathcal{D}]_{\textup{BV}([0,T]; \R)} \leq C[F]_{\textup{BV}([0, T]; H^{1/2}(\partial U))} + (1 + \mu_+\vee \mu_-)[{\bf 1}_{\Omega(\cdot)}]_{\textup{BV}([0, T]; L^1(U))} 
\end{equation}
where $C$ depends on the geometry of $\partial U$, on the lower bound $\varepsilon_0$ for $F$, and on $\|F\|_{L^\infty_t H^{1/2}_x}$.

\end{proof}

\begin{theorem}\label{t.mm-energy}
    Suppose $\partial U$ is sufficiently regular, $F\in \mathrm{BV}([0, T]; H^{1/2}(\partial U))$, and $\Omega_0$ is given. Then any minimizing movements solution for $F$ and $\Omega_0$ is an energy solution.
\end{theorem}
\begin{proof}

We argue as in \cite{FeldmanKimPozar}*{Theorem 6.2}. First, we derive BV estimates to obtain compactness of the discrete scheme and energy convergence. Passing to a convergent subsequence, we can then check global stability and the energy dissipation inequality.

The main ingredient in the BV estimates is a discrete version of the energy dissipation inequality. Using the notation of Definition \ref{def:mmsln}, we want to develop a competitor argument to compare $u^{t_{k-1}}_\delta$ to $u^{t_k}_\delta$. As in Lemma \ref{l.energy-time-BV-reg}, we use Lemma \ref{lem:domain_size_estimate} to fix a domain $\omega$ on which any globally stable profile for $F(t)$ at any time $t$ is bounded away from 0. We then take $h^{s,t}\in H^1(U)$ as in \eqref{eq:additive_boundary_data_corrector}, so that $u^{t_{k-1}}_\delta + h^{t_{k-1}, t_{k}}$ is admissible as a competitor to $u^{t_k}_\delta$, in the minimizing movements step from $u^{t_{k-1}}_\delta$. Moreover, we will always assume that $\delta$ is sufficiently small based on the uniform continuity of $F$ in $L^\infty(\partial U)$ to ensure that $\|F(t_k) - F(t_{k-1})\|_{L^\infty(\partial U)} \leq \varepsilon_1$ holds, so that $\Omega(u^{t_{k-1}}_\delta) = \Omega(u^{t_{k-1}}_\delta + h^{t_{k-1}, t_k})$. Then the minimality of $u^{t_k}_\delta$ implies that
\begin{align*}
   \mathcal{J}[u^{t_k}_\delta] + \Diss[u^{t_{k-1}}_\delta, u^{t_k}_\delta] 
   &\leq \mathcal{J}[u^{t_{k-1}}_\delta + h^{t_{k-1}, t_k}]
   \\&= \mathcal{J}[u^{t_{k-1}}_\delta] + \int_U 2\nabla u^{t_{k-1}}_\delta \cdot \nabla h^{t_{k-1}, t_k}\, dx + \|\nabla h^{t_{k-1}, t_k}\|_{L^2(U)}^2
\end{align*}
Integrating by parts in the $\nabla u^{t_{k-1}}_\delta \cdot \nabla h^{t_{k-1}, t_k}$ term, we get
\[ 2\int_{\partial U} (F(t_k) - F(t_{k-1}))\partial_\nu u^{t_{k-1}}_\delta dx  \]
which is recognizable as the time-discrete form of the pressure integral. We conclude the following discrete-time energy dissipation inequality:
\begin{align}\label{eq:one_step_dissipation}
    \Diss[u^{t_{k-1}}_\delta, u^{t_k}_\delta] &\leq \mathcal{J}[u^{t_{k-1}}_\delta] - \mathcal{J}[u^{t_{k}}_\delta] + 2\int_{\partial U} (F(t_k) - F(t_{k-1}))\partial_\nu u^{t_{k-1}}_\delta dx\notag
    \\&+ C\|F(t_k) - F(t_{k-1})\|_{H^{1/2}(\partial U)}^2
\end{align}
We remark that since $F\in \mathrm{BV}([0,T]; H^{1/2}(\partial U))\cap C([0,T], H^{1/2}(\partial U))$, we have $\sum_k \|F(t_k) - F(t_{k-1})\|_{H^{1/2}(\partial U)}^2 \to 0$ as $\delta\to 0$, so the error term vanishes in the continuous-time limit. Thus, by summing \eqref{eq:one_step_dissipation} and employing the arguments of Proposition \ref{p.energy-solutions-temporal-reg} with minor modifications for the time-discrete case, we obtain bounds on $\one_{\Omega(u_\delta(t))} \in \mathrm{BV}([0, T]; L^1(\R^d))$ and $\mathcal{J}[u_\delta(t)], \mathcal{D}[u_\delta(t)], \langle 2\dot{F}(t), \frac{\partial u}{\partial n}\rangle_{L^2(\partial U)}\in \mathrm{BV}([0, T]; \R)$ which are uniform as $\delta\to 0$. By Helly's selection principle, we can then take a subsequence along which all of these converge pointwise in $t$. By the same arguments as in \cite{FeldmanKimPozar}, we subsequently get uniform convergence of $u_\delta(t)$ as $\delta \to 0$ and Hausdorff convergence of $\Omega(u_\delta(t))$ as $\delta\to 0$, and one can check that the limits of all of these objects are consistent with the same limit function $u(t)$.

Now, we must check that this $u(t)$ is indeed an energy solution. We have immediately that $u(t)\in F(t) + H^1_0(U)$ using the time regularity of $F$ and the convergence estimates discussed above. We also get the energy dissipation inequality by summing \eqref{eq:one_step_dissipation} and applying the convergence estimates to each term in the inequality. Thus, it remains only to verify global stability.

Let $v\in F(t) + H^1_0(U)$ be a competitor to $u(t)$ in the definition of global stability; note that without loss of generality, we may assume $v$ itself is globally stable. Referring back to the discrete scheme, let $k(\delta)$ denote the index such that $t\in [t_{k(\delta)}, t_{k(\delta) + 1})$. We once again employ the perturbation from \eqref{eq:additive_boundary_data_corrector} to form a competitor argument. Note that $v + h^{t, t_{k(\delta)}}$ is admissible as a competitor to $u^{t_{k(\delta)}}_\delta$, which is globally stable by minimality in the discrete scheme. Assuming once again that $\delta$ is sufficiently small to avoid the perturbation modifying the positivity set, we get
\[ \mathcal{J}[u^{t_{k(\delta)}}_\delta] \leq \mathcal{J}[v + h^{t, t_{k(\delta)}}] + \Diss[u^{t_{k(\delta)}}_\delta, v] \]
By definition of $k(\delta)$, we may rewrite this as
\[ \mathcal{J}[u_\delta(t)] \leq \mathcal{J}[v] + \Diss[u_\delta(t), v] + C\|F(t) - F(t_{k(\delta)})\|_{H^{1/2}(\partial U)} \]
Then using the convergence of $\mathcal{J}[u_\delta(t)]$ and $\one_{\Omega(u_\delta(t))}$, we can pass to the limit to get 
\[ \mathcal{J}[u(t)] \leq \mathcal{J}[v] + \Diss[u(t), v] \]

\end{proof}

\bibliography{minimizing-movements-uniqueness.bib}

\end{document}